\documentclass[10pt]{amsart}

\usepackage{amssymb,tikz}
\usepackage{amsthm}
\usepackage{amsmath, a4wide}
\usepackage{amsbsy}
\usepackage[all]{xy}
\usepackage{bm}
\usepackage{hyperref}

\newtheorem{thm}{Theorem}
\newtheorem*{theorem}{Theorem}
\newtheorem*{proposition}{Proposition}

\begin{document}

\title[]{Directional Poincar\'{e} inequalities \\along mixing flows}

\author[]{Stefan Steinerberger} \thanks{Department of Mathematics, Yale University, 10 Hillhouse Avenue, New Haven, CT 06511, e-mail: \textsc{stefan.steinerberger@yale.edu}}
\address{Department of Mathematics, Yale University, 10 Hillhouse Avenue, New Haven, CT 06511, USA}
\email{stefan.steinerberger@yale.edu}
\subjclass[2010]{37A30 (primary), and 49Q20 (secondary)} 
\begin{abstract} We provide a refinement of the Poincar\'{e} inequality on the torus $\mathbb{T}^d$: there
exists a set $\mathcal{B} \subset \mathbb{T}^d$ of directions such that for every
$\alpha \in \mathcal{B}$ there is a $c_{\alpha} > 0$ with
$$ \|\nabla f\|_{L^2(\mathbb{T}^d)}^{d-1} \| \left\langle \nabla f, \alpha \right\rangle\|_{L^2(\mathbb{T}^d)} \geq c_{\alpha}\|f\|_{L^2(\mathbb{T}^d)}^{d} \qquad
\mbox{for all}~f\in H^1(\mathbb{T}^d)~\mbox{with mean 0.}$$
The derivative $\left\langle \nabla f, \alpha \right\rangle$ does not detect any oscillation in directions orthogonal to $\alpha$, however,
for certain $\alpha$ the geodesic flow in direction $\alpha$ is sufficiently mixing to compensate for that defect.
On the two-dimensional torus $\mathbb{T}^2$ the inequality
holds for $\alpha = (1, \sqrt{2})$ but fails for $\alpha = (1,e)$. Similar results should hold at a great level of generality on very general domains.
\end{abstract}
\maketitle

\section{Introduction and main result}
\subsection{Introduction} The classical Poincar\'{e} inequality on the torus $\mathbb{T}^d$ states
$$ \| \nabla f\|_{L^2(\mathbb{T}^d)} \geq \| f\|_{L^2(\mathbb{T}^d)}$$
for functions $f \in H^1(\mathbb{T}^d)$ with vanishing mean. A natural interpretation is that a function with small derivatives cannot substantially deviate from its mean on
a set of large measure. The purpose of this paper is to derive a substantial improvement; we first state the main result.

\begin{thm}[Directional Poincare inequality] There exists a set $\mathcal{B} \subset \mathbb{T}^d$  such that for every
$\alpha \in \mathcal{B}$ there is a $c_{\alpha} > 0$ so that
$$ \|\nabla f\|_{L^2(\mathbb{T}^d)}^{d-1} \| \left\langle \nabla f, \alpha \right\rangle\|_{L^2(\mathbb{T}^d)} \geq c_{\alpha}\|f\|_{L^2(\mathbb{T}^d)}^{d}$$
for all $f\in H^1(\mathbb{T}^d)$ with mean 0. If $d \geq 2$, then $\mathcal{B}$ is uncountable but Lebesgue-null.
\end{thm}
The exponents are optimal. The proof is simple and based on elementary properties of Fourier series -- we believe it to be of great interest to understand under which conditions comparable inequalities exist on a general Riemannian manifold $(M,g)$ equipped with a suitable vector field.

\begin{center}
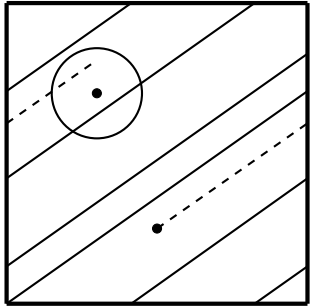
\begin{figure}[h!]
\centering
\begin{tikzpicture}[scale=4]
\draw[ultra thick] (0,0) -- (1,0);
\draw[ultra thick] (1,1) -- (1,0);
\draw[ultra thick] (1,1) -- (0,1);
\draw[ultra thick] (0,1) -- (0,0);
\filldraw (0.3,0.7) circle (0.015cm);
\filldraw (0.5,0.25) circle (0.015cm);
\draw [thick] (0.3,0.7) circle (0.15cm);
\draw[dashed, thick] (0.5,0.25) -- (1,0.6);
\draw[dashed, thick] (0,0.6) -- (0.3,0.81);
\draw[thick] (0,0) -- (1,1/1.4142);
\draw[thick] (0.4142,1) -- (0,1/1.4142);
\draw[thick] (0.4142,0) -- (1,0.59/1.4142);
\draw[thick] (0,0.59/1.4142) -- (0.824208,1);
\draw[thick] (0.824208,0) -- (1,0.1243) ;
\draw[thick] (0,0.1243) -- (1,0.831416);
\end{tikzpicture}
\caption{A well-mixing flow transports (dashed) every point relatively quickly to a neighborhood of every other point.}
\end{figure}
\end{center}

On the torus, the inequality has strong ties with number theory and can be easily derived at the cost of invoking highly nontrivial results (Schmidt's result on badly approximable numbers, the Khintchine theorem). One remarkable feature is that the inequality holds for a set of Lebesgue measure 0 which shows it to be very delicate (however, as is explained below, slightly weaker statements seem to be very robust). A natural interpretation of the inequality seems to be the following: given two nearby points $x,y \in \mathbb{T}^d$ for which $f(x) \gg f(y)$,
the classical Poincar\'{e} inequality will detect a large gradient between them. The term $| \left\langle \nabla f, \alpha \right\rangle |$ might not
detect the large gradient but following the ergodic vector field will relatively quickly lead to a neighborhood of $y$.  A priori being in a
neighborhood might not imply much because there could be still local oscillations on the scale of the neighborhood, however, since
we also invoke a power of $\|\nabla f\|_{L^2(\mathbb{T}^d)}$, this controls the measure of the set on which local oscillations have a strong effect. This heuristic suggests strongly that similar inequalities should hold at a much greater level of generality. We discuss and prove some natural variants in the last section.

\subsection{Open problems} It would be of great interest to understand to which extent such inequalities can be true in a more general setup. It is not even clear to us whether comparable inequalities hold in $L^p(\mathbb{T}^d)$. Generally, for suitable vector fields $Y$ on suitable Riemannian manifolds $(M,g)$ it seems natural to ask whether there exists an inequality of the type
$$ \|\nabla f\|_{L^p(M)}^{1-\delta} \| \left\langle \nabla f, Y \right\rangle f \|^{\delta}_{L^p(M)} \geq c\|f\|_{L^p(M)}^{}$$
for some $\delta > 0$ and all $f\in W^{1,p}(M)$ with mean 0. The parameter $\delta$ can be expected to be related to the mixing properties of the flow -- it is difficult to predict what the \textit{generic}
behavior on a fixed manifold might be (say, for a smooth perturbation of the flat metric on the torus). On $\mathbb{T}^2$ we can rephrase the Khintchine theorem \cite{ki} as a statement about generic behavior for the flat metric.
\begin{theorem}[Khintchine, equivalent] For every $\delta < 1/2$, the set of $\alpha \in \mathbb{T}^2$ for which there exists a $c_{\alpha} >0$ such that
$$  \forall f \in H^1(\mathbb{T}^2)\qquad \int_{\mathbb{T}^2}{f(x)dx} = 0 \implies \|\nabla f\|_{L^2(\mathbb{T}^2)}^{1-\delta} \| \left\langle \nabla f, \alpha \right\rangle f \|^{\delta}_{L^2(\mathbb{T}^2)} \geq c_{\alpha}\|f\|_{L^2(\mathbb{T}^2)}^{}$$
has full measure.
\end{theorem}
This suggests $\delta < 1/2$ as a natural threshold that might be achievable by other two-dimensional examples -- however, since the inequality is extremely delicate on $\mathbb{T}^2$, the manifolds on which the inequality holds with $\delta = 1/2$ might actually be very rare. One would expect topological ingredients to arise on the sphere $\mathbb{S}^{d}$ equipped with a nontrivial vector field: the hairy
ball theorem dictates that any smooth vector field vanishes for even $d$ and this will necessitate a change of scaling in the inequality since a function $f$ could be concentrated around the point in which the vector field vanishes. Furtermore, while not every nonvanishing vector field on $\mathbb{S}^3$ has to have a closed orbit (i.e. Seifert's conjecture is false), many of them do -- this puts topological restrictions on what directional Poincar\'{e} inequalities are possible (since one could set a function to be constant along a periodic orbit and decaying extremely quickly away from it). However, there should be a variety of admissible inequalities on the flat infinite
cylinder $(M,g) = (\mathbb{R} \times \mathbb{T}^{d-1}, \mbox{can})$ and this could be a natural starting point for future investigations.

\section{Proof of the Statements}
\subsection{Outline of the argument.} The proof of the classical Poincar\'{e} inequality on the torus is a one-line argument if one expands in Fourier series and uses $\hat{f}(0) = \int_{\mathbb{T}^d}{f}= 0$ since
$$ \| \nabla f\|^2_{L^2(\mathbb{T}^d)} = (2\pi)^d \sum_{k \in \mathbb{Z}^d \atop k \neq 0}{|k|^2 |a_k|^2} \geq  (2\pi)^d\sum_{k \in \mathbb{Z}^d \atop k \neq 0}{|a_k|^2} = \| f\|^2_{L^2(\mathbb{T}^d)}.$$
The argument also highlights the underlying convexity of the quadratic form.
Our proof will be a direct variation of that result and uses the observation that
\begin{align*}
 \| \left\langle \nabla f, \alpha \right\rangle \|^2_{L^2(\mathbb{T}^d)} &= \left\| \left\langle \sum_{k \in \mathbb{Z}^d}{a_k k e^{i k \cdot x}}, \alpha \right\rangle \right\|^2_{L^2(\mathbb{T}^d)} \\
&= \left\|   \sum_{k \in \mathbb{Z}^d}{a_k \left\langle k, \alpha \right\rangle e^{i k \cdot x} } \right\|^2_{L^2(\mathbb{T}^d)}\\
&=  (2\pi)^d \sum_{k \in \mathbb{Z}^d}{|a_k|^2 |\left\langle k, \alpha \right\rangle|}^2.
\end{align*}
This Fourier multiplier is not uniformly bounded away from 0 and will even vanish for certain $k \in \mathbb{Z}^d$ if the entries of $\alpha$ are
not linearly independent over $\mathbb{Q}$. If the entries of $\alpha$ are linearly independent over $\mathbb{Q}$, then
the Fourier multiplier is always nonnegative but we have no quantitative control on its decay (see below for an example). However, if we restrict 
the support of the Fourier coefficients to be within $\left\{k \in \mathbb{Z}^d: |k| \leq R \right\}$, then trivially
$$  \| \left\langle \nabla f, \alpha \right\rangle \|^2_{L^2(\mathbb{T}^d)} \geq \left(\inf_{k \in \mathbb{Z}^d \atop |k| \leq R}{|\left\langle k, \alpha \right\rangle|^2}\right) \|f\|_{L^2(\mathbb{T}^d)}^2.$$
The term in the bracket clearly has great significance in the study of geometry of numbers and has been studied for a long time. It suffices for us
to apply the results and use the additional $\left\| \nabla f\right\|_{L^2(\mathbb{T}^d)}$ expression to ensure that a fixed proportion of the $L^2-$mass
is contained within a suitable ball of frequeny space on which to apply the argument.

\subsection{Number theoretical properties.}
We now discuss subtleties of the inequality in greater detail: it is merely the classical Poincar\'{e} inequality for $d=1$. Letting $d=2$ with $\alpha = (1,0)$ yields 
$$ \|\nabla f\|_{L^2(\mathbb{T}^2)} \| \partial_x f \|_{L^2(\mathbb{T}^2)} \geq c\|f\|_{L^2(\mathbb{T}^2)}^{2} \qquad \mbox{which is obviously false,}$$
because $f$ might be constant along the $x-$direction and vary along the $y-$direction. More generally, the inequality fails for any
$\alpha$ with entries linearly dependent over $\mathbb{Q}$ and the functions $\sin{(k_1 x + k_2 y)}$ for any $k_1, k_2 \in \mathbb{Z}^2$ with $\left\langle (k_1, k_2), \alpha \right\rangle = 0$ serve
as counterexamples. The next natural example is $\alpha = (\sqrt{2},1).$
Suppose $f \in C^{\infty}(\mathbb{T})$ and
 $$ \left\| \left\langle \nabla f, (\sqrt{2},1) \right\rangle\right\|_{L^2(\mathbb{T}^2)} =0.$$ $f$ is constant
along the flow of the vector field $(\sqrt{2}, 1)$ but every orbit is dense and thus $f \equiv 0$. 
This is true for any vector with entries that are linearly independent over $\mathbb{Q}$, however, it is not enough
to prove the inequality itself: it fails for $(1,e)$ on $\mathbb{T}^2$ despite linear independence.
A simple construction for $d=2$ shows that linear independence of the entries of $\alpha$ is not enough:
 let  
$$ \alpha = \left(1, \sum_{n=1}^{\infty}{\frac{1}{10^{n!}}}\right) \sim (1, 0.110001\dots)$$
where the arising number, Liouville's constant, is known to be irrational. If we set
$$ f_N(x,y) = \sin{\left(10^{N!}\left( \sum_{n=1}^{N}{\frac{x}{10^{n!}}} - y\right)\right)}, $$
then 
$$\|f_N\|^2_{L^2(\mathbb{T}^2)} = 2\pi^2 \qquad \mbox{and} \qquad \|\nabla f_N\|_{L^2(\mathbb{T}^2)} \leq 6\cdot 10^{N!}$$
while
$$ \left\| \left\langle \nabla f_N, \alpha \right\rangle \right\|_{L^2(\mathbb{T}^2)} = \sqrt{2\pi^2}\left(\sum_{n=N+1}^{\infty}{\frac{10^{N!}}{10^{n!}}}\right)  \ll 10^{-2\cdot N!} \qquad \mbox{for}~N \geq 3.$$

\subsection{An explicit example.} The inequalities are not only sharp with respect to exponents, they are actually sharp \textit{on all frequency scales}. This is in stark contrast to classical
Poincar\'{e}-type inequality which tend to be sharp for one function (the ground state of the underlying physical system): here, we can exclude all functions having Fourier support in the
set $\left\{ \xi: |\xi| \leq N \right\}$ for arbitrarily large $N$ and still find functions for which the inequality is sharp (up to a constant). We explain this in greater detail for $d=2$ with the admissible
direction given by the golden ratio
$$ \alpha = \left(1, \frac{1 + \sqrt{5}}{2}\right) \in \mathcal{B}.$$
 Consider the sequence of functions given by
$$ f_n(x,y) = \sin{\left(F_{n+1}x - F_{n}y\right)},$$
where $F_n$ is the $n-$th Fibonacci number. An explicit computation shows that
$$
  \|\nabla f_n\|_{L^2(\mathbb{T}^2)} \left\| \partial_x f_n +  \frac{1 + \sqrt{5}}{2} \partial_y f_n \right\|_{L^2(\mathbb{T}^2)} 
=\sqrt{\frac{F_{n+1}^2}{F_n^2} +1}  \left|\frac{F_{n+1}}{F_n} - \frac{1 + \sqrt{5}}{2} \right|F_n^2 \|f_n\|_{L^2(\mathbb{T}^2)}^2 
$$
A standard identity for Fibonacci numbers gives that
$$ \lim_{n \rightarrow \infty}{\left|\frac{F_{n+1}}{F_n} - \frac{1 + \sqrt{5}}{2} \right|F_n^2} = \frac{1}{\sqrt{5}},$$
which implies that
$$ \lim_{n \rightarrow \infty}{   \|\nabla f_n\|_{L^2(\mathbb{T}^2)} \left\| \partial_x f_n +  \frac{1 + \sqrt{5}}{2} \partial_y f_n \right\|_{L^2(\mathbb{T}^2)}  \|f_n\|_{L^2(\mathbb{T}^2)}^{-2} =  
\sqrt{\frac{5+\sqrt{5}}{10}}} =  \frac{|\alpha|}{\sqrt{5}}.$$
Since these functions $f_n$ have their Fourier transform supported on 4 points in $\mathbb{Z}^2$ and since $F_{n+1}/F_{n} \rightarrow (1+\sqrt{5})/2 < 2$, we can conclude that
every dyadic annulus in Fourier space contains an example for which the inequality is sharp (up to a constant). Put differently, our inequality is close to being attained on every frequency scale. 
This sequence of $f_n$ has the advantage of simultaneously showing that the following statement is sharp.

\begin{proposition} Let $d=2$ and $\alpha \in \mathbb{T}^2$ be any vector for which 
$$ \|\nabla f\|_{L^2(\mathbb{T}^2)} \| \left\langle \nabla f, \alpha \right\rangle\|_{L^2(\mathbb{T}^2)} \geq c_{\alpha}\|f\|_{L^2(\mathbb{T}^2)}^{2}$$
holds for all $f \in H^1(\mathbb{T}^2)$ with mean 0. Then the constant satisfies
$$ c_{\alpha} \leq  \frac{|\alpha|}{\sqrt{5}}.$$
\end{proposition}
Up to certain transformation, the example above is essentially the only example for which the inequality is tight: normalizing $\alpha = (1, \beta)$, the example shows that the inequality is sharp for $\beta = (1+\sqrt{5})/2$ and there are only countably many other $\beta$ for which it is sharp (that can be explicitely given). For all other numbers the upper bound could be improved to $c_{\alpha} \leq |\alpha|/\sqrt{8}$. Removing
yet another countable set of exceptional directions, we could replace $\sqrt{8}$ by $\sqrt{221}/5$ and the process could be continued (this follows from classical results about the structure of the Markov spectrum, see \cite{mark}).

\subsection{Badly approximable systems of linear forms.} We now introduce the relevant results from number theory. Let $L_1, \dots, L_{\ell}: \mathbb{Z}^d \rightarrow \mathbb{R}$ be defined as
\begin{align*}
L_1(\textbf{x}) &= \alpha_{11}x_1 + \dots \alpha_{1d}x_d = \left\langle \alpha_1, \textbf{x} \right\rangle \\
\dots &\qquad  \qquad   \qquad   \qquad   \qquad  \\
L_{\ell}(\textbf{x}) &= \alpha_{\ell 1}x_1 + \dots \alpha_{ \ell d}x_d = \left\langle \alpha_{\ell}, \textbf{x} \right\rangle 
\end{align*}
The relevant question is whether it is possible for all $\ell$ expressions
to be very close to an integer. Using $\left\| \cdot \right\|:\mathbb{R} \rightarrow [0, 1/2]$ to denote the distance to the closest integer, the
pigeonhole principle implies the existence of infinitely many $\textbf{x} \in \mathbb{Z}^d$ with
$$ \max(\| L_1(\textbf{x}) \|, \dots, \| L_{\ell}(\textbf{x}) \|) \leq  \left( \max(|x_1|, \dots, |x_d|) \right)^{-\frac{d}{\ell}}.$$
Dirichlet's theorem cannot be improved in the sense that there actually exist badly approximable vectors
$\alpha_1, \dots, \alpha_{\ell}$ such that for some $c >0$ and all $\textbf{x} \in \mathbb{Z}^d$
$$ \max(\| L_1(\textbf{x}) \|, \dots, \| L_{\ell}(\textbf{x}) \|)  \geq  c\left( \max(|x_1|, \dots, |x_d|) \right)^{-\frac{d}{\ell}}.$$
The existence of such elements was first shown by Perron \cite{per}. Khintchine \cite{ki} has shown that the $\ell d-$dimensional Lebesgue measure
 of such tuples $(\alpha_1, \dots, \alpha_{\ell})$ is 0 and Schmidt \cite{schm} has proven
that their Hausdorff dimension is $\ell d$. 

\subsection{Proof of Theorem 1.} \begin{proof} We will prove the statement explicitely for the following set: for any $(d-1)-$dimensional badly approximable vector $\alpha_{d-1}$, 
consider the linear form $L:\mathbb{Z}^{d-1} \rightarrow \mathbb{R}$ given by
$$ L(\textbf{x}) := \left\langle \alpha_{d-1}, \textbf{x} \right\rangle$$
and
$$ \alpha = (1, \alpha_{d-1}).$$
Recall that $\left\| \cdot \right\|:\mathbb{R} \rightarrow [0, 1/2]$ denotes the distance to the closest integer and is trivially 1-periodic. Let now $k \in \mathbb{Z}^{d}$
with $k \neq 0$. If $k$ vanishes on all but the first component, then
$$ |\left\langle \alpha, k \right\rangle | = |k_1| \geq 1.$$
If $k$ does not vanish on all but the first component, then
$$
 |\left\langle \alpha, k \right\rangle | =  |k_1 + L((k_2, \dots, k_d))| \geq || L((k_2, \dots, k_d)) || 
\geq \frac{c_{}}{\max(|k_2|, \dots, |k_{d-1}|)^{d-1}} \geq \frac{c}{|k|^{d-1}},
$$

where $c > 0$ is some constant coming from the fact that $\alpha_{d-1}$ is badly approximable. Let now
$$ f = \sum_{k \in \mathbb{Z}^d}{a_k e^{ i k \cdot x}} \in H^1(\mathbb{T}^d)$$
and note that $a_0 = 0$ because $f$ has mean value 0. We have
\begin{align*}  \| \left\langle \nabla f, \alpha \right\rangle\|^2_{L^2(\mathbb{T}^d)} &=
 (2\pi)^d\sum_{k \in \mathbb{Z}^d}{|a_k|^2 | \left\langle k, \alpha \right\rangle |^2} \geq 
 c^2(2\pi)^d\sum_{k \in \mathbb{Z}^d}{\frac{|a_k|^2}{|k|^{2d-2}}}.
\end{align*}
It is easy to see that
$$ \sum_{|k| \geq 2\frac{\|\nabla f\|_{L^2(\mathbb{T}^d)}}{\| f\|_{L^2(\mathbb{T}^d)}}}{|a_k|^2} \leq \frac{\|f\|^2_{L^2(\mathbb{T}^d)}}{2}$$
because the opposite inequality would imply that
\begin{align*}
  \| \nabla f\|^2_{L^2(\mathbb{T}^d)} &= \sum_{k \in \mathbb{Z}^d}{|k|^2|a_k|^2} \geq  \sum_{|k| \geq 2\frac{\|\nabla f\|_{L^2(\mathbb{T}^d)}}{\| f\|_{L^2(\mathbb{T}^d)}}}{|k|^2|a_k|^2} \\
&\geq  4\frac{\|\nabla f\|^2_{L^2(\mathbb{T}^d)^2}}{\| f\|^2_{L^2(\mathbb{T}^d)^2}}\sum_{|k| \geq 2\frac{\|\nabla f\|_{L^2(\mathbb{T}^d)^2}}{\| f\|_{L^2(\mathbb{T}^d)^2}}}{|a_k|^2} 
\geq 2  \| \nabla f\|^2_{L^2(\mathbb{T}^d)},
\end{align*}
which is absurd. Altogether, we now have 
\begin{align*} 
 \| \left\langle \nabla f, \alpha \right\rangle\|^2_{L^2(\mathbb{T}^d)} \geq c^2 (2\pi)^d \sum_{k \in \mathbb{Z}^d}{\frac{|a_k|^2}{|k|^{2d-2}}} &\geq c^2 (2\pi)^d\sum_{k \in \mathbb{Z}^d \atop |k| \leq 2\frac{\|\nabla f\|_{L^2(\mathbb{T}^d)}}{\| f\|_{L^2(\mathbb{T}^d)}}}{\frac{|a_k|^2}{|k|^{2d-2}}} \\
&\geq \frac{c^2 (2\pi)^d}{2^{2d-2}} \frac{\| f\|^{2d-2}_{L^2(\mathbb{T}^d)}}{\| \nabla f\|^{2d-2}_{L^2(\mathbb{T}^d)}}  \sum_{|k| \leq 2\frac{\|\nabla f\|_{L^2(\mathbb{T}^d)}}{\| f\|_{L^2(\mathbb{T}^d)}}}{|a_k|^2} \\
&\geq \frac{c^2(2\pi)^d}{2^{2d-2}} \frac{\| f\|^{2d-2}_{L^2(\mathbb{T}^d)}}{\| \nabla f\|^{2d-2}_{L^2(\mathbb{T}^d)}}  \frac{\|f\|^2_{L^2(\mathbb{T}^d)}}{2}.
\end{align*}
Rearranging gives the result.
\end{proof}
We remark that a classical insight of Liouville allows to give a completely self-contained proof in the most elementary case. If we pick $\alpha = (\sqrt{2},1)$,
then the only information required to make the above argument work is the existence of a $c > 0$ such that
$$ | \left\langle k, \alpha \right\rangle| = |k_1\sqrt{2} + k_2| \geq \frac{c}{|k|} \qquad \mbox{for all}~k \in \mathbb{Z}^2.$$
However, this follows at once with $c=1/3$ from 
$$ 1 \leq |2 k_1^2 - k_2^2| = |(\sqrt{2}k_1 - k_2)(\sqrt{2}k_1 + k_2)| \leq 3|k||(\sqrt{2}k_1 + k_2)|.$$
A similar argument works for any $(1, \alpha)$ with $\alpha$ algebraic over $\mathbb{Q}$ (Liouville's theorem). More generally, a classical
characterization of badly approximable numbers in one dimension as those numbers with a bounded continued fraction expansion implies that our
proof works for
$$ \alpha = (1, \beta) \in \mathbb{R}^2$$
if $\beta$ has a bounded continued fraction expansion. A theorem of Lagrange (see e.g. \cite{kin}) implies that this is always the case if $\beta$ is a quadratic irrational.
Moreover, this characterization is sharp on $\mathbb{T}^2$: if $\beta$ has an unbounded continued fraction expansion, then the inequality is not true for $(1, \beta)$ and the sequence
$$ f_n(x) = e^{2 \pi i k_n \cdot x}$$
with $k_n = $(numerator, -denominator) of rational approximations of $\beta$ coming from the continued fraction expansion will serve as a counterexample. A very interesting special case
is Euler's continued fraction formula for $e$ (see, e.g. \cite{io}), which implies that $e$ has an unbounded continued fraction expansion and that the inequality with
$\alpha = (1,e)$ fails on $\mathbb{T}^2$.

\subsection{Proof of the Proposition.}
\begin{proof} The direction $\alpha$ has to have both entries different from 0. We use a classical result of Hurwitz \cite{hur} which guarantees the existence
of infinitely many $k \in \mathbb{Z}^2$ with
$$ \left| \frac{\alpha_1}{\alpha_2} - \frac{k_1}{k_2} \right| \leq \frac{1}{\sqrt{5}k_2^2}.$$
For any such $(k_1, k_2)$, this can be rewritten as
$$ |\alpha_1 k_2 - \alpha_2 k_1| \leq \frac{|\alpha_2|}{\sqrt{5} |k_2|}.$$
We now consider $f(x) = e^{2\pi i k \cdot x}$. Simple computation yields
$$  \|\nabla f\|_{L^2(\mathbb{T}^2)}  \| \left\langle \nabla f, \alpha \right\rangle\|_{L^2(\mathbb{T}^2)}    \leq \frac{1}{\sqrt{5}} \frac{|\alpha_2||k|}{|k_2|}\|f\|_{L^2(\mathbb{T}^2)}^2.$$
However, as $|k| \rightarrow \infty$, we have that $k_1/k_2 \rightarrow \alpha_1/\alpha_2$ and thus
$$ \frac{|\alpha_2||k|}{|k_2|} = |\alpha_2| \sqrt{\frac{k_1^2}{k_2^2}+1} \longrightarrow  |\alpha_2| \sqrt{\frac{\alpha_1^2}{\alpha_2^2}+1} = |\alpha|.$$
\end{proof}
The constant in the result of Hurwitz is sharp for the golden ratio $\alpha = \alpha_1/\alpha_2 = (1+\sqrt{5})/2 = \phi$. Moreover, it is known (see e.g. \cite{cassels}) that for every 
$$\alpha \in \mathbb{R} \setminus \mathbb{Q} \quad \mbox{which is not of the form} \quad \frac{a \phi + b}{c \phi + d} \qquad a,b,c,d \in \mathbb{Z} \quad |ad-bc|=1,$$
the constant $\sqrt{5}$ could be replaced by $\sqrt{8}$. Our example showing the sharpness of the Proposition using Fibonacci numbers was therefore, in some sense, best possible.

\subsection{Fractional derivatives.} As is obvious from the proof, fine properties of the derivative did not play a prominent role, indeed, the proof really only requires an understanding
of how fast the induced Fourier multiplier grows. This allows for various immediate generalizations. We introduce pseudodifferential operators $P(D)$ on $H^s(\mathbb{T}^d)$
via
$$ P(D) \sum_{k \in \mathbb{Z}^d}{a_k e^{i k \cdot x}} :=  \sum_{k \in \mathbb{Z}^d}{a_k P(k) e^{i k \cdot x}} $$
the same proof can immediately be applied as long as $|P(k)| \rightarrow \infty$ if $|k| \rightarrow \infty$. One example on $\mathbb{T}^2$ would be that for $\alpha \in \mathcal{B}$ and all $s>0$
$$ \|\nabla^s f\|^{\frac{1}{s}}_{L^2(\mathbb{T}^{2})} \| \left\langle \nabla f, \alpha_i \right\rangle\|_{L^2(\mathbb{T}^2)} \geq c_{\alpha}\|f\|^{1+\frac{1}{s}}_{L^2(\mathbb{T}^2)}$$
which is again sharp by the same reasoning as above. The following variant was proposed by Raphy Coifman: if we define 
$$ D^s \sum_{k \in \mathbb{Z}^d}{a_k e^{i k \cdot x}} :=  \sum_{k \in \mathbb{Z}^d}{a_k k |k|^{s-1} e^{i k \cdot x}},$$
then one can always observe substantial fluctuations in the $(d-1)-$th derivative along the flow 
$$ \left\| \left\langle D^{d} f, \alpha \right\rangle \right\|_{L^2(\mathbb{T}^d)} \geq c_{\alpha}\|f\|_{L^2(\mathbb{T}^d)}.$$

\subsection{Several ergodic directions.} We can also derive a statement for more than one ergodic direction. The same heuristics as above still apply: the main difference is that incorporating a control in more than one ergodic direction poses additional restrictions and requires less global control in the sense that a proportionately smaller power of $\|\nabla f\|_{L^2(\mathbb{T}^d)}$ is necessary. 

\begin{thm} Let $1 \leq \ell \leq d-1$. Then there exists a set $\mathcal{B}_{\ell} \in (\mathbb{T}^d)^{\ell}$ such that for every $(\alpha_1, \alpha_2, \dots, \alpha_{\ell}) \in \mathcal{B_{\ell}}$ there is a $c_{\alpha} > 0$ with
$$ \|\nabla f\|^{d-1}_{L^2(\mathbb{T}^{d})} \left(\sum_{i=1}^{\ell}{ \| \left\langle \nabla f, \alpha_i \right\rangle\|_{L^2(\mathbb{T}^d)}}\right)^{\ell} \geq c_{\alpha}\|f\|^{d-1+\ell}_{L^2(\mathbb{T}^d)}$$
for all $f\in H^1(\mathbb{T}^d)$ with mean 0.
\end{thm}
\begin{proof} We consider $\ell \leq d-1$ 
vectors $\beta_1, \beta_2, \dots, \beta_{\ell}$ from $\mathbb{T}^{d-1}$ with $\ell$ associated linear forms $L_i:\mathbb{Z}^{d-1} \rightarrow \mathbb{R}$
via $L_i = \left\langle \beta_i, \textbf{x} \right\rangle$ such that they form a system of badly approximable linear forms and set
\begin{align*}
\alpha_1 &= (1, \beta_1) \\
\dots& \\
\alpha_{\ell} &= (1, \beta_{\ell})
\end{align*}
The same reasoning as before (distinguishing between $k$ vanishing outside of the first component or not) implies again for every single $1 \leq i \leq \ell$
$$ |\left\langle \alpha_i, k \right\rangle | =  |k_1 + L_i((k_2, \dots, k_d))| \geq  || L_i((k_2, \dots, k_d)) ||$$
from which we derive that whenever $k$ is not concentrated on the first component
$$ \sum_{i=1}^{\ell}{  |\left\langle \alpha_i, k \right\rangle | } \geq   \max(\| L_1(k) \|, \dots, \| L_{\ell}(k) \|)  \geq  c\left( \max(|k_2|, \dots, |k_d|) \right)^{-\frac{d-1}{\ell}}.$$
If $k$ is concentrated on the first component, we get a bound of $\ell$, which is much larger.
For
$$ f = \sum_{k \in \mathbb{Z}^d \atop k \neq 0}{a_k e^{i k \cdot x}}$$
a simple computation shows that
\begin{align*} \sum_{i=1}^{\ell}{ \| \left\langle \nabla f, \alpha_i \right\rangle\|^2_{L^2(\mathbb{T}^d)}} &=
 (2\pi)^d\sum_{k \in \mathbb{Z}^d}{|a_k|^2 \left(\sum_{i=1}^{\ell}{ | \left\langle k, \alpha_i \right\rangle |^2}\right)} \\
&\geq_{}
 c^2\sum_{k \in \mathbb{Z}^d}{|a_k|^2  \left( \max(|k_1|, \dots, |k_d|) \right)^{-\frac{2(d-1)}{\ell}}} \\
&\geq c^2  \sum_{k \in \mathbb{Z}^d}{|a_k|^2 |k|^{- \frac{2(d-1)}{\ell}}}.
\end{align*}
The rest of the argument proceeds as before; finally, we recall that any two norms in finite-dimensional vector spaces are equivalent and thus, up to some absolute constants depending only on $\ell$,
 $$ \left(\sum_{i=1}^{\ell}{ \| \left\langle \nabla f, \alpha_i \right\rangle\|_{L^2(\mathbb{T}^d)}}\right)^{\ell} \sim_{\ell} \left(\sum_{i=1}^{\ell}{ \| \left\langle \nabla f, \alpha_i \right\rangle\|^2_{L^2(\mathbb{T}^d)}}\right)^{\frac{\ell}{2}}$$
and the result follows.
\end{proof}

\subsection{The Hausdorff dimension.} As is obvious from the proof, we are not so much interest in the distance to the lattice but care more about the distance to the origin. It seems that
there is ongoing research in that direction \cite{dick, huss2, huss}, which is concerned with establishing bounds on the dimension of the set
$$ \max(| L_1(\textbf{x}) |, \dots, | L_{\ell}(\textbf{x}) |)  \geq  c\left( \max(|x_1|, \dots, |x_d|) \right)^{-\frac{d}{\ell}+1},$$
where $|\cdot|$ is the absolute value on $\mathbb{R}$. For any such system of linear forms given by $\alpha_1, \dots, \alpha_{\ell}$ satisfying that inequality, we can improve Theorem 2 with the same proof to
$$ \|\nabla f\|^{d-\ell}_{L^2(\mathbb{T}^{d})} \left(\sum_{i=1}^{\ell}{ \| \left\langle \nabla f, \alpha_i \right\rangle\|_{L^2(\mathbb{T}^d)}}\right)^{\ell} \geq c_{\alpha}\|f\|^{d}_{L^2(\mathbb{T}^d)}$$
for all $f\in H^1(\mathbb{T}^d)$ with mean 0. We also remark the following simple proposition (the essence of which is contained in \cite{huss2}).
\begin{proposition} Let $d \geq 2$. $\alpha \in \mathbb{R}^d$ is admissible in Theorem 1 if and only if there exists $\lambda \in \mathbb{R}$ such that
$$ \alpha = \left(\lambda, \lambda\beta\right)$$ 
and $\beta$ is badly approximable linear form in $\mathbb{R}^{d-1}$. 
\end{proposition}
Using the result of Schmidt \cite{schm}, we see that the Hausdorff dimension of the set of badly approximable vectors in $\mathbb{R}^{d-1}$ is $d-1$ and the construction increases the dimension by 1.
The Hausdorff dimension of the set of admissible vectors in Theorem 1 is therefore $d$. The argument is indirectly contained in the earlier proofs.

\subsection{Variants.} This subsection is concerned with inequalities of the type
$$ \|\nabla f\|_{L^2(\mathbb{T}^2)}^{1-\delta} \| \left\langle \nabla f, \alpha \right\rangle f \|^{\delta}_{L^2(\mathbb{T}^2)} \geq c\|f\|_{L^2(\mathbb{T}^2)}^{}$$
for some $0 < \delta \leq 1/2$. The case $\delta =1/2$ was discussed above and following the same arguments immediately imply that $\delta > 1/2$ is impossible. However, the threshold $\delta = 1/2$ is also sharp.
\begin{theorem}[Khintchine] For every $\delta < 1/2$, the set of $\alpha \in \mathbb{T}^2$ for which there exists a $c_{\alpha} > 0$ such that
$$ \|\nabla f\|_{L^2(\mathbb{T}^2)}^{1-\delta} \| \left\langle \nabla f, \alpha \right\rangle f \|^{\delta}_{L^2(\mathbb{T}^2)} \geq c\|f\|_{L^2(\mathbb{T}^2)}^{}$$
holds for all $f \in L^2(\mathbb{T}^2)$ with mean 0 has full measure.
\end{theorem}
Another celebrated result in Diophantine approximation is the Thue-Siegel-Roth theorem stating that for every irrational algebraic number $\alpha$ and every $\varepsilon > 0$ we have 
$$ \left| \alpha - \frac{p}{q} \right| \geq \frac{c_{\alpha}}{q^{2+\varepsilon}}$$
for some $c_{\alpha} > 0$. This immediately implies, along the same lines as above, that for every vector of the form $\alpha = (1, \beta)$ with $\beta$ being an irrational algebraic number and every
$\varepsilon > 0$
$$ \|\nabla f\|_{L^2(\mathbb{T}^2)}^{1/2 + \varepsilon} \| \left\langle \nabla f, \alpha \right\rangle f \|^{1/2 - \varepsilon}_{L^2(\mathbb{T}^2)} \geq c_{\alpha, \varepsilon}\|f\|_{L^2(\mathbb{T}^2)}^{}.$$
Recall that the inequality does \textit{probably} not hold for $\alpha = (1,\pi)$ because $\pi$ is \textit{probably} not badly
approximable. However, there are weaker positive results. A result of Salikhov \cite{sali} implies the existence of a $c_{} > 0$ such that
$$ \left| \pi - \frac{p}{q} \right| \geq \frac{c_{}}{q^{8}}.$$
Repeating again the same argument as above, we can use this to derive
$$ \|\nabla f\|_{L^2(\mathbb{T}^2)}^{7/8} \| \left\langle \nabla f, (1,\pi) \right\rangle  \|^{1/8}_{L^2(\mathbb{T}^2)} \geq c\|f\|_{L^2(\mathbb{T}^2)}^{}.$$
Similarly, a result of Marcovecchio \cite{marc} shows
$$ \|\nabla f\|_{L^2(\mathbb{T}^2)}^{13/18} \| \left\langle \nabla f, (1,\log{2}) \right\rangle  \|^{5/18}_{L^2(\mathbb{T}^2)} \geq c\|f\|_{L^2(\mathbb{T}^2)}^{}$$
and similar results are available for other numbers (i.e. $\pi^2, \log{3}, \zeta(3), \dots$).
\medskip

\section{Uniform Poincar\'{e} inequalities via diffusion}
There exists a closely related problem, where we can try to substitute a directional derivative by a random process (that points roughly in the direction of its mean but may also behave differently). Let $\mathbb{T} = \mathbb{R} \setminus \mathbb{Z}$ be the one-dimensional torus $\mathbb{T}$, let $f \in L^p(\mathbb{T})$
and consider the averaging operator $A_t: L^p(\mathbb{T}) \rightarrow L^p(\mathbb{T})$ given by
$$ (A_tf)(x) = \frac{1}{2t}\int_{-t}^{+t}{f(x+y) dy}.$$
An iterative application of $A_t$ corresponds to a time-discrete diffusion.
These objects have strong and natural ties to Poincar\'{e} inequalities (see, for example, the recent book by Bakry, Gentil and Ledoux \cite{led}).
One natural question is whether $\|f - A_t f\|_{L^p(\mathbb{T})}$ can be small if $f$ is not constant. 
Motivated by a problem of Schechtman, this was recently investigated by Nayar \& Tkocz \cite{nay}. They studied the problem in the general
case of convolution with a $\mathbb{T}-$valued random variable $Y$ which for some $\ell \in \mathbb{N}$ 
$$ Y_1 + Y_2 + \dots + Y_{\ell} \quad \mbox{ has a nontrivial absolutely continuous part.}$$

\begin{thm}[Nayar \& Tkocz] There exists a constant $c > 0$ depending only on $Y$ such that
for all $f \in L^p(\mathbb{T})$, $1 \leq p \leq \infty$, with $\int_{\mathbb{T}}{f} = 0$
$$ \| f(x) - \mathbb{E}f(x + tY) \|_{L^p(\mathbb{T})}  \geq c t^2 \| f \|_{L^p(\mathbb{T})}.$$
\end{thm}

We believe the inequality to be of quite some interest; the example above corresponds to $Y$ being uniformly distributed on $\mathbb{T}$ but the setting is clearly more general than that.
It is easy to see that the inequality is sharp:
for a fixed $C^2(\mathbb{T})-$function $f$ and $Y$ being again the uniform distribution it follows from a Taylor expansion that
$$ \lim_{t \rightarrow 0}{\frac{\|f - A_t f\|_{L^p}}{t^2}} = \frac{1}{6}\|f''\|_{L^p}.$$
Using $C^{2}(\mathbb{T}) \hookrightarrow L^p(\mathbb{T})$, the statement immediately implies a Poincar\'{e} inequality
$$ \|f''\|_{L^p(\mathbb{T})} \geq c\|f\|_{L^p(\mathbb{T})}$$
with a constant $c > 0$ that does not depend on $p$.
There is a natural variant: note that for any $f \in C^1$ the case $\mathbb{E}Y \neq 0$ yields a first order contribution
$$ \lim_{t \rightarrow 0}{\frac{\mathbb{E}(f(x + t Y)) - f(x)}{t} } = (\mathbb{E}Y)f'(x).$$
This suggests that the cases $\mathbb{E}(Y) = 0$ and $\mathbb{E}(Y) \neq 0$ behave differently: from
the point of view of diffusion, the case $\mathbb{E}Y \neq 0$ turns the averaging operator $\mathbb{E}f(x+tY)$
into an operator with drift, which allows for a faster exploration of the torus. We prove that this is indeed the case.
 
\begin{thm} \label{mu} Assume $\mathbb{E}Y \neq 0$. Then there exists a constant $c > 0$ depending only on $Y$ such that
for all $f \in L^p(\mathbb{T})$, $1 \leq p \leq \infty$ with $\int_{\mathbb{T}}{f} = 0$
$$ \| f(x) - \mathbb{E}f(x + tY) \|_{L^p(\mathbb{T})}\geq c t\| f \|_{L^p(\mathbb{T})}.$$
\end{thm}
It is easy to see that the statement cannot be improved: consider $Y \sim \mathcal{U}([0,1/2])$ to be
the uniform distribution on $[0,1/2]$ and let $f = \chi_{[0,1/2]} - \chi_{[1/2, 1]}$. Then $\|f\|_{L^p(\mathbb{T})} = 1$
and for all $1 \leq p < \infty$
$$ \| f(x) - \mathbb{E}f(x + tY) \|_{L^p(\mathbb{T})} \sim t.$$
By taking 
the limit $t \rightarrow 0$, we obtain
$$ \|f'\|_{L^p(\mathbb{T})} \geq c\|f\|_{L^p(\mathbb{T})}$$
with a constant $c > 0$ that does not depend on $p$. We will first prove the statements and
then comment on the relation with the material presented above.

\subsection{The proof of Nayar \& Tkocz.} We start by giving a summary of the proof given in \cite{nay}
since our argument for Theorem 1 is based on that approach. We write the averaging operator
 $$ f(x) \rightarrow \mathbb{E}f(x + tY) \qquad \mbox{as a convolution} \qquad f \rightarrow f*\mu.$$
If $\|f - f*\mu\|_{L^p(\mathbb{T})}$ is small, then $f*\mu \sim f$ and repeating
the argument suggests that 
$$\|f*\mu - f*\mu*\mu\|_{L^p(\mathbb{T})} \qquad \mbox{should be small as well.}$$
Since $\mu$ comes from a probability measure, Young's convolution inequalities implies that
$$\|f*\mu - f*\mu*\mu\|_{L^p(\mathbb{T})} = \|(f - f*\mu)*\mu\|_{L^p(\mathbb{T})} \leq  \|f - f*\mu\|_{L^p(\mathbb{T})}.$$
 The next
ingredient is the triangle inequality allowing us to compare

\begin{align*}
\| f - f*\mu\|_{L^p(\mathbb{T})} &\geq \frac{1}{n}\left(\| f - f*\mu\|_{L^p(\mathbb{T})} + \| f*\mu - f*\mu^2 \|_{L^p(\mathbb{T})} + \dots +  \| f*\mu^{n-1} - f*\mu^n \|_{L^p(\mathbb{T})} \right)   \\
&\geq \frac{1}{n} \| f - f*\mu^n\|_{L^p(\mathbb{T})},
\end{align*}
where
$$ \mu^{k} := \underbrace{\mu * \mu * \dots * \mu}_{k~\mbox{times}}$$
is the $k-$fold convolution of $\mu$ with itself.
The final ingredient of the argument is the following: if $n$ is sufficiently large, then the $n-$fold convolution $\mu^n$ is very well-behaved: its absolutely continuous
part satisfies a uniform lower bound. However, if $\mu^n \geq c$ in a pointwise sense and $f$ has mean 0, then \cite[Lemma 2]{nay}
$$ \|f*\mu\|_{L^p(\mathbb{T})} \leq (1-c)\|f\|_{L^p(\mathbb{T})}$$
and therefore
$$ \| f - f*\mu\|_{L^p(\mathbb{T})} \geq \frac{1}{n}\| f - f*\mu^n\|_{L^p(\mathbb{T})} \geq \frac{1}{n} \left( \| f \|_{L^p(\mathbb{T})} - \| f*\mu^n\|_{L^p(\mathbb{T})}\right) \geq \frac{c}{n}\| f \|_{L^p(\mathbb{T})}.$$
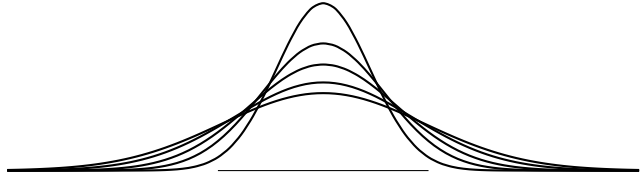
\begin{figure}[h!]
\begin{center}
\begin{tikzpicture}[scale=0.7]
  \draw[] (-2,0) -- (2,0);
  \draw[scale=2,domain=-3:3,samples=55,smooth,variable=\x,black, thick] plot ({\x},{(sqrt(4/sqrt(2.5)))*exp(-4/sqrt(2.5)*(\x)*(\x))});
  \draw[scale=2,domain=-3:3,samples=55,smooth,variable=\x,black, thick] plot ({\x},{(sqrt(4/sqrt(54.5)))*exp(-4/sqrt(54.5)*(\x)*(\x))});
  \draw[scale=2,domain=-3:3,samples=55,smooth,variable=\x,black, thick] plot ({\x},{(sqrt(4/sqrt(7.5)))*exp(-4/sqrt(7.5)*(\x)*(\x))});
  \draw[scale=2,domain=-3:3,samples=55,smooth,variable=\x,black, thick] plot ({\x},{(sqrt(4/sqrt(15.5)))*exp(-4/sqrt(15.5)*(\x)*(\x))});] 
 \draw[scale=2,domain=-3:3,samples=55,smooth,variable=\x,black, thick] plot ({\x},{(sqrt(4/sqrt(32.5)))*exp(-4/sqrt(32.5)*(\x)*(\x))});
\end{tikzpicture}
\caption{Flow of an expanding heat kernel along a geodesic.}
\end{center}
\end{figure}

It remains to find the right value of $n$ for which $\mu^n \geq c$. It is easy to see that this generally scales like $n \sim 1/t^2$ for sufficiently nice $\mu$
and it remains to show that $\mu^n$ is sufficiently nice for $n$ large enough assuming that $\mu^{\ell}$ has an absolutely continuous part for some $\ell$.
Summarizing, the heat kernel flattens and will eventually cover the entire torus; this requires roughly $1/t^2$ convolutions
which follows from the usual scaling that the variance of $\mu^k$ scales as $\sqrt{k}$. It is not surprising that a 
drift should allow for this exploration of the torus to happen at a different time scale.

\subsection{Random variables with drift: proof of Theorem 4}
\begin{proof} The argument of Nayar \& Tkocz used the contraction property to establish
$$ \|f - f*\mu^{k}\|_{L^p(\mathbb{T})} \leq k\|f - f*\mu\|_{L^p(\mathbb{T})}.$$
We will work with an averaged version of that statement: summing from $k=1, \dots, n$, we get
$$ \frac{1}{n} \sum_{k=1}^{n}{\|f - f*\mu^{k}\|_{L^p(\mathbb{T})}} \leq n\|f - f*\mu\|_{L^p(\mathbb{T})}.$$
Another application of the triangle inequality implies
$$ \left\| f - f*\left(\frac{1}{n}\sum_{k=1}^{n}\mu^{k}\right)\right\|_{L^p(\mathbb{T})} \leq
\frac{1}{n} \sum_{k=1}^{n}{\|f - f*\mu^{k}\|_{L^p(\mathbb{T})}}  \leq  n\|f-f*\mu\|_{L^p(\mathbb{T})}.$$
The conclusion of the argument proceeds in the same way (though, obviously by a different method): if we can conclude that 
$$ \frac{1}{n}\sum_{k=1}^{n}\mu^{k} \geq c \quad \mbox{pointwise, then} \quad \|f-f*\mu\|_{L^p(\mathbb{T})} \geq \frac{c}{n}\| f\|_{L^p(\mathbb{T}^d)}.$$

It remains to show that the statement is asymptotically true for $n \sim t^{-1}$ to conclude the result. We show that a result
from renewal theory can now be applied relatively easily.
Define $Y_{\mathbb{R}}$ to be $Y$ on $[-1/2, 1/2] \subset \mathbb{R}$ (with corresponding measure $\mu_{\mathbb{R}}$) and assume furthermore w.l.o.g. that $\mathbb{E}Y_{\mathbb{R}} > 0$.
We first note that if some convolution $$Y^{\ell} = Y_1 + Y_2 + \dots + Y_{\ell}$$ on the torus has an absolutely continuous component, then so will $Y_{\mathbb{R}}^{\ell}$ on $\mathbb{R}$: this follows
immediately from the Radon-Nikodym theorem and the fact that the singular measure is supported on a set of measure 0 (moreover, $Y_{\mathbb{R}}^{l}$
is supported on $[-\ell/2, \ell/2]$). This allows us to apply a result of C. Stone \cite{stone}: suppose $\mu_{\mathbb{R}}$ 
\begin{enumerate}
\item has a finite $m-$th moment 
\item $\mu_{\mathbb{R}}^{\ell}$ has a nontrivial continuous component for some $\ell \in \mathbb{N}$ 
\item and if furthermore $\mathbb{E} \mu_{\mathbb{R}} > 0$ 
\end{enumerate}
then the measure $\sum_{k=1}^{\infty}\mu^{k}$ can be written as
$$\sum_{k=1}^{\infty}\mu^{k} = \nu_1 + \nu_2,$$
where $\nu_1$ is absolutely continuous and has a density $p(x)$ satisfying
$$ p(x) = \begin{cases} o(|x|^{-m}) \qquad &\mbox{as}~x \rightarrow -\infty \\
(\mathbb{E}Y_{\mathbb{R}})^{-1}+ o(|x|^{-m}) \qquad &\mbox{as}~x \rightarrow \infty
\end{cases}$$
and $\nu_2$ is a finite measure such that all integer moments exist.

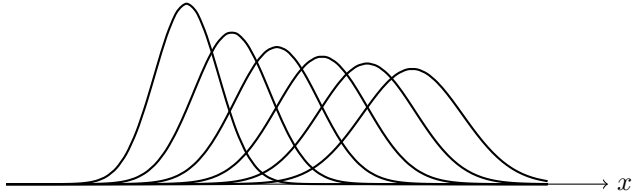
\begin{figure}[h!]
\begin{center}
\begin{tikzpicture}[scale=0.8]
  \draw[->] (-2,0) -- (7,0) node[right] {$x$};
  \draw[scale=1.5,domain=-2:4,samples=55,smooth,variable=\x,black, thick] plot ({\x},{2*exp(-4*\x*\x)});
  \draw[scale=1.5,domain=-2:4,samples=55,smooth,variable=\x,black, thick] plot ({\x},{(sqrt(4/sqrt(2))*exp(-4/sqrt(2)*(\x-0.5)*(\x-0.5))});
  \draw[scale=1.5,domain=-2:4,samples=55,smooth,variable=\x,black, thick] plot ({\x},{(sqrt(4/sqrt(3))*exp(-4/sqrt(3)*(\x-1)*(\x-1))});
  \draw[scale=1.5,domain=-2:4,samples=55,smooth,variable=\x,black, thick] plot ({\x},{(sqrt(4/sqrt(4))*exp(-4/sqrt(4)*(\x-1.5)*(\x-1.5))});
  \draw[scale=1.5,domain=-2:4,samples=55,smooth,variable=\x,black, thick] plot ({\x},{(sqrt(4/sqrt(5))*exp(-4/sqrt(5)*(\x-2)*(\x-2))});
  \draw[scale=1.5,domain=-2:4,samples=55,smooth,variable=\x,black, thick] plot ({\x},{(sqrt(4/sqrt(6))*exp(-4/sqrt(6)*(\x-2.5)*(\x-2.5))});
\end{tikzpicture}
\caption{Heat kernels with drift.}
\end{center}
\end{figure}

 The result even says slightly more and allows to conclude exponential decay
of $\nu_2$ as well as $p(x) - 1/(\mathbb{E}Y_{\mathbb{R}})$ from exponential decay of $\mu_{\mathbb{R}}$ (which is something that is the case in our setting since
$\mu_{\mathbb{R}}$ is compactly supported but we will not use this fact). We now localize that statement (so as to avoid summing up to infinity). 
Note that $\mu_{\mathbb{R}}^{k}$ has a large proportion of its mass centered in an interval of length $\sim \sqrt{k}$ centered around $k (\mathbb{E}\mu_{\mathbb{R}})$. 
This easily implies that, for $n$ sufficiently large,
$\sum_{k=1}^{n}\mu^{k} $
has an absolutely continuous part with a density $p(x)$ satisfying
$$ p(x) \geq \frac{1}{2}\frac{1}{\mathbb{E} Y_{\mathbb{R}}} \qquad \mbox{for all} \qquad x_0 \leq x \leq \frac{n}{2}\mathbb{E}Y_{\mathbb{R}},$$
where $x_0$ is a universal number depending only on $Y$ (essentially the threshold beyond which the Stone asymptotic becomes
effective). We remark that much sharper asymptotics could be derived and one could prove the same result for
$$x_0 \leq x \leq (n-C\sqrt{n})\mathbb{E}Y_{\mathbb{R}} \qquad \mbox{for some fixed}~C>0~\mbox{sufficiently large}$$
but this is not required. 
Applying this to $t Y_{\mathbb{R}}$ is very simple because the
entire setup scales nicely under dilations. If we define $\mu_t$ as the probability measure of $tY_{\mathbb{R}}$. Then
$\mu_t$ has an absolutely continuous component $p_t(x)$ satisfying
$$ p_t(x) \geq \frac{1}{2}\frac{1}{t}\frac{1}{\mathbb{E} Y_{\mathbb{R}}} \qquad \mbox{for all} \qquad tx_0 \leq x \leq \frac{n}{2}t\mathbb{E}Y_{\mathbb{R}}.$$
Let now $n =8/(t \mathbb{E} Y_{\mathbb{R}}) + 2x_0/(t \mathbb{E} Y_{\mathbb{R}})$. We can conclude that
$$ p_t(x) \geq \frac{1}{2}\frac{1}{t}\frac{1}{\mathbb{E} Y_{\mathbb{R}}} \qquad \mbox{on an interval of length at least 2 on}~\mathbb{R}.$$
Since $n \sim 1/t$, this now implies that
 $$ \frac{1}{n}\sum_{k=1}^{n}{\mu_{t,\mathbb{R}}^{k}} \qquad \mbox{has an absolutely continuous component with density}~\geq c$$
for some $c>0$ (that could be described in terms of quantities above) on an interval of length at least 2 on $\mathbb{R}$.
By projecting down on the torus, we have that 
 $$ \frac{1}{n}\sum_{k=1}^{n}{\mu_t^{k}} \qquad \mbox{has an absolutely continuous component with density}~\geq c$$
and this implies the desired result by the reasoning described above.
\end{proof}

\subsection{Concluding Remarks}
Let us try to understand the arguments above in a more general context; the natural framework seems to be that of a compact Lie group and random variables
for which the central limit theorem is valid. We are interested in the largest possible function $h:\mathbb{R}_{+} \rightarrow \mathbb{R}_{+}$ such that there exists a constant $c > 0$ depending only on $Y$ with
the property that
for all $f \in L^p$, $1 \leq p \leq \infty$ satisfying $\int_{}{f} = 0$
$$ \| f(x) - \mathbb{E}f(x + tY) \|_{L^p}\geq h(t)\| f \|_{L^p}.$$
The argument of Nayar \& Tkocz implies that if $\mathbb{E}Y = 0 \in \mathbb{T}^d$ the best answer is given by $h(t) \sim t^2$ and this follows from the fact that the characteristic time for the heat kernel to effectively cover the entire torus is of scale $t^{-2}$. If $\mathbb{E} Y \neq 0$, then the situation becomes more interesting. 
\begin{center}
\begin{figure}[h!]
\centering
\begin{tikzpicture}[scale=4]

\draw[ultra thick] (0,0) -- (1,0);
\draw[ultra thick] (1,1) -- (1,0);
\draw[ultra thick] (1,1) -- (0,1);
\draw[ultra thick] (0,1) -- (0,0);
\draw[thick] (0,0) -- (1,1/1.4142);
\draw [thick] (0.2,0.1414) circle (0.03 cm);
\draw [thick] (0.4,0.2828) circle (0.042 cm);
\draw [thick] (0.6,0.4242) circle (0.051 cm);
\draw [thick] (0.8,0.5656) circle (0.06 cm);
\draw [thick] (0,0.707) circle (0.067 cm);
\draw [thick] (0.2,0.8485) circle (0.073 cm);
\draw [thick] (0.4,0.9899) circle (0.079 cm);
\draw [thick] (0.6,0.131371) circle (0.084 cm);
\draw [thick] (0.8,0.2727) circle (0.09 cm);
\draw [thick] (0,0.4142) circle (0.094 cm);
\draw [thick] (0.2,0.5556) circle (0.0994 cm);
\draw [thick] (0.4,0.6970) circle (0.1039 cm);
\draw [thick] (0.6,0.8384) circle (0.1081 cm);
\draw [thick] (0.8,0.9798) circle (0.1122 cm);
\draw [thick] (0,0.1213) circle (0.1161 cm);
\draw [thick] (0.2,0.2627) circle (0.12 cm);
\draw [thick] (0.4,0.4041) circle (0.123 cm);
\draw [thick] (0.6,0.5455) circle (0.1272 cm);
\draw [thick] (0.8,0.687) circle (0.13072 cm);

\draw[thick] (0.4142,1) -- (0,1/1.4142);
\draw[thick] (0.4142,0) -- (1,0.59/1.4142);
\draw[thick] (0,0.59/1.4142) -- (0.824208,1);
\draw[thick] (0.824208,0) -- (1,0.1243) ;
\draw[thick] (0,0.1243) -- (1,0.831416);

\end{tikzpicture}
\caption{Location of the mass of the first few convolutions $\mu^{\ell}$.}
\end{figure}
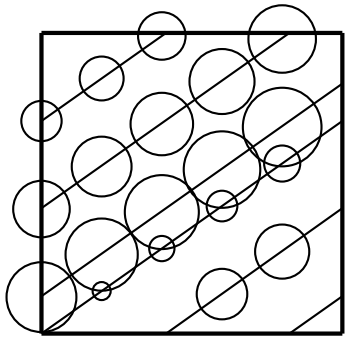
\end{center}
Using standard approximations, we may assume that $\mu_t^k$ ($\mu_t$ again being the probability measure 
of $tY$) has most of its mass around $t \mathbb{E}Y$ in a radius of scale $\sim \sqrt{t}$. This may be best understood as a moving and expanding bubble. The
problem is to understand the number of bubbles $N$ required to have
$$ \frac{1}{N} \sum_{k=1}^{N}{\mu^k} \geq c.$$

On $\mathbb{T}^2$ simple volumetric considerations imply that $N \gtrsim t^{-1}$. However, it is also not difficult to see that $N \lesssim t^{-1}$ is false, even for badly approximable $\mathbb{E}Y$: this would 
imply $h(t) \gtrsim t$, which in turn would imply a Poincar\'{e} type inequality
$$  \left\| \left\langle \nabla f, \alpha \right\rangle \right\|_{L^p} \gtrsim \|f\|_{L^p},$$
which we know to fail even for badly approximable directions. Another interesting example is the sphere, where one cannot improve on $h(t) \gtrsim t^2$
because every geodesic flow is periodic. It seems interesting to understand how the precise quantitative behavior in $t$ depend on approximation properties of $\mathbb{E} Y$ 
and, more generally, how such problems behave in different geometries.\\

\textbf{Acknowledgement.} I am grateful to Raphy Coifman for various discussions and to him, Yves Meyer and Jacques Peyri\`{e}re for their encouragement.

\end{document}